\documentclass[11pt]{amsart}

\usepackage{amsfonts,amssymb,amscd,amsmath,enumerate,verbatim,calc}
\usepackage{amsmath}
\usepackage{amssymb}
\usepackage{amscd}
\usepackage{color}

\usepackage[T1]{fontenc}

\usepackage{graphicx}

\def\RMN#1{\uppercase\expandafter{\romannumeral#1}}

\newcommand{\proj}{\operatorname{Proj}}
\newcommand{\length}{\operatorname{length}}
\newcommand{\spec}{\operatorname{Spec}}

\newcommand{\rank}{\operatorname{rank}}

\newcommand{\subq}{_{\Bbb Q}}

\newtheorem{Theorem}{Theorem}[section]
\newtheorem{Lemma}[Theorem]{Lemma}
\newtheorem{Corollary}[Theorem]{Corollary}
\newtheorem{Proposition}[Theorem]{Proposition}
\newtheorem{Conjecture}[Theorem]{Conjecture}
\newtheorem{Claim}[Theorem]{Claim}

\newtheorem{Remark}[Theorem]{Remark}
\newtheorem{Example}[Theorem]{Example}
\newtheorem{Assumption}[Theorem]{Assumption}
\newtheorem{Question}[Theorem]{Question}

\newcommand{\new}[1]{{\color{blue} #1}}

\DeclareMathOperator{\Pic}{Pic}

\DeclareMathOperator{\depth}{depth}

\DeclareMathOperator{\Cl}{Cl}

\newcommand{\CC}{\mathbf{C}}

\newcommand{\ZZ}{\mathbb{Z}}
\newcommand{\QQ}{\mathbb{Q}}
\newcommand{\RR}{\mathbb{R}}

\begin{document} 
\title[Boundary of Cohen-Macaulay cone]{Boundary and shape of Cohen-Macaulay cone}

\author{Hailong Dao}
\address{Department of Mathematics\\
University of Kansas\\
Lawrence, KS 66045-7523, USA}
\email{hdao@ku.edu}

\author{Kazuhiko Kurano}
\address{Department of Mathematics\\
School of Science and Technology\\
Meiji University\\
Higashimita 1-1-1, Tama-ku, Kawasaki-shi 214-8571, Japan}
\email{kurano@isc.meiji.ac.jp}
\date{}
\begin{abstract}
Let $R$ be a Cohen-Macaulay local domain. In this paper we study the cone of Cohen-Macaulay modules inside the Grothendieck group of finitely generated $R$-modules modulo 
numerical equivalences, introduced in \cite{CK}.  We prove a result about the boundary of this cone for Cohen-Macaulay domain admitting de Jong's alterations, and use it to derive some corollaries on finiteness of  isomorphism classes of maximal Cohen-Macaulay  ideals. Finally, we explicitly compute the Cohen-Macaulay cone for certain isolated hypersurface singularities defined by $\xi\eta - f(x_1, \ldots, x_n)$. 
\end{abstract}

\thanks{The first author is partially supported by NSF grant DMS 1104017. 
The second author is partially supported by JSPS KAKENHI Grant 24540054.} 

\maketitle

\section{Introduction}\label{intro}
Let $R$ be a Noetherian local ring and $G_0(R)$ the Grothendieck group of  finitely generated $R$-modules. Using   Euler characteristic of perfect complexes with finite length homologies (a generalized version of Serre's intersection multiplicity pairings), one could define the notion of {\it numerical equivalence} on  $G_0(R)$ as in \cite{K23}. See Section \ref{sec2} for precise definitions. When $R$ is the local ring at the vertex of an affine cone over a smooth projective variety $X$, this notion can be {deeply} related to that of numerical equivalences on the Chow group of $X$
as in \cite{K23} and \cite{RS}. Let $\overline{G_0(R)}$ be the Grothendieck group of $R$ modulo numerical equivalences. 

A simple result in homological algebra tells us that if $M$ is maximal Cohen-Macaulay (MCM), the Euler characteristic function will always be positive. Thus, maximal Cohen-Macaulay modules all survive in $\overline{G_0(R)}$, and it makes sense to talk about the cone of Cohen-Macaulay modules inside $\overline{G_0(R)}_{\Bbb R}$:

\[
C_{CM}(R) = \sum_{M: {\rm MCM}} {\Bbb R}_{\ge 0}[M] \subset \overline{G_0(R)}_{\Bbb R}.
\]
	The definition of this cone and some of its basic  properties was given in \cite{CK}.  Understanding the Cohen-Macaulay cone is quite challenging, as it encodes a lot of non-trivial information about both the category of maximal Cohen-Macaulay modules and the local intersection theory on $R$. 

For example, let's consider the hypersurface $k[[x,y,u,v]]/(xy-uv)$. Then, the main result of \cite{DHM} tells us that $\overline{G_0(R)}_{\Bbb R}$ is two dimensional, and one can pick a basis consisting of $[R]$ (vertical)  and $[R/(x,u)]$ (horizontal). Since the maximal Cohen-Macaulau modules are completely classified (\cite{K}), one can compute the Cohen-Macaulay cone of the  as shown in Figure \ref{fig1}. It is the blue region between the rays from the origin to the points represented by the modules $(x,u)$ and $(x,v)$. Later in the paper (Section \ref{CMcones} we shall  compute these cones for a big class of hypersurface singularities.
\begin{figure}\label{fig1}
    \centering
    \includegraphics[width=200pt,height=200pt,scale=0.4]{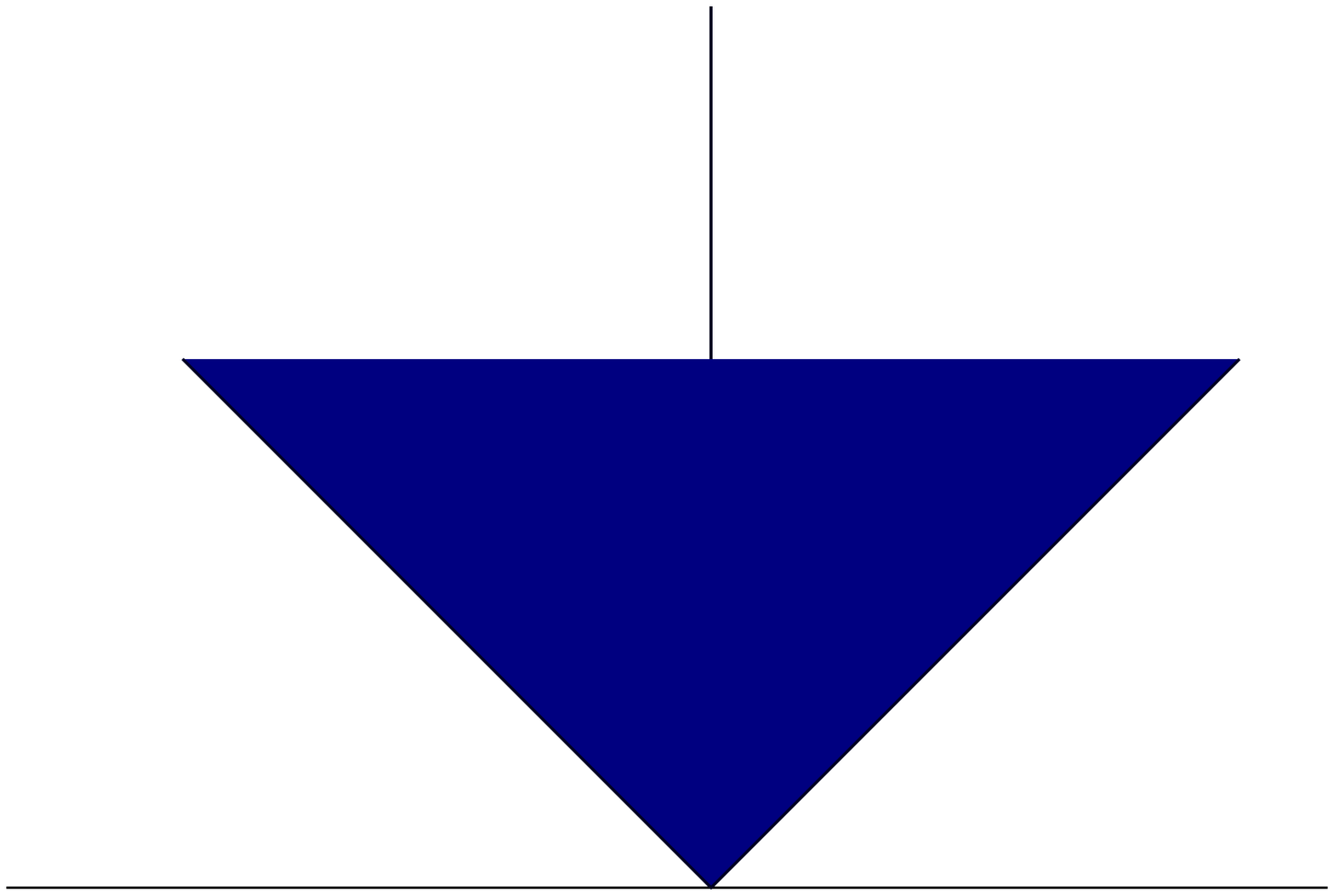}
    \caption{CM cone for $k[[x,y,u,v]]/(xy-uv)$}
    \label{fig:awesome_image}
\end{figure}

One of the main technical results of this paper, Theorem \ref{2.1}, asserts that if $R$ is a Cohen-Macaulay domain admitting de Jong's alterations (a weak version of resolution of singularity), the closure of $C_{CM}(R)$ intersects with the hyperplane of torsion modules only at the origin. This  has some surprising consequences, which we now describe. 

The first consequence is that given any integer $r$, the set of isomorphism classes of  maximal Cohen-Macaulay modules of rank $r$ up to numerical equivalence is finite (Theorem \ref{1.2}). This is interesting as results from Cohen-Macaulay representation theory tells us that the set of these modules in any given rank is typically very big (see \cite{LW} or \cite{Y} for some nice introduction to this topic. When $r=1$, our results  confirm a number of cases when the set of Cohen-Macaulay ideals are actually finite. This is raised in a question attributed to Hochster, which we discuss in details  in Section \ref{sechoch}.

{In Section \ref{4.1} we study when the map $A_{d-1}(R) \longrightarrow \overline{A_{d-1}(R)}$ has finite kernel. This turns out to be a rather subtle question, even when $R$ is the localization at the vertex of a cone over a smooth projective variety. However, we are able to prove that the map is an isomorphism for most isolated hypersurface singularities of dimension $3$ (Proposition \ref{1.6}).}

Finally, in Section \ref{CMcones} we describe explicitly the Cohen-Macaulay cone for a  class of local hypersurfaces defined by a power series of the form $\xi\eta - f(x_1, \ldots, x_n)$. These are first examples of non-trivial Cohen-Macaulay cones in higher dimensions over rings not of finite Cohen-Macaulay type, and the mere fact that they can actually be computed even for this special class of singularities is quite encouraging. 

{In an upcoming paper, we shall apply the results here to study asymptotic behavior of systems of ideals. 

We thank Hubert Flenner, Ryo Takahashi, Charles Vial and Yuji Yoshino for many helpful conversations. We also thank Eleonore Faber for helping us to translate Karroum's thesis (\cite{K})}.

\section{Notations and preliminary results}\label{sec2}
\setcounter{equation}{0}

We always assume that a Noetherian local ring in this paper
is a homomorphic image of a regular local ring.

For a Noetherian local ring $R$,
we denote the Grothendieck group of  finitely generated $R$-modules
(resp.\ the Chow group of $R$) by $G_0(R)$ (resp.\  $A_*(R)$).
We refer the reader to \cite{F} for basic facts on them.
Let us recall the definition of numerical equivalence on $G_0(R)$ (see \cite{K23}).

For a bounded finite $R$-free complex ${\Bbb F}.$ with homologies of finite length,
we define
\[
\chi_{{\Bbb F}.} : G_0(R) \longrightarrow {\Bbb Z}
\]
to be \[
\chi_{{\Bbb F}.}([M]) = \sum_{i}(-1)^i\ell(H_i({\Bbb F}.\otimes_RM)) .
\]
We say that a cycle $\alpha$ in $G_0(R)$ is {\em numerically equivalent to
$0$} if $\chi_{{\Bbb F}.}(\alpha) = 0$ for any bounded finite $R$-free complex ${\Bbb F}.$ with homologies of finite length.
In the same way, we say that a cycle $\beta$ in $A_*(R)$
is {\em numerically equivalent to
$0$} if ${\rm ch}({\Bbb F}.)(\beta) = 0$ for any above ${\Bbb F}.$,
where ${\rm ch}({\Bbb F}.)$ is the localized Chern character
(see Chapter~17 in \cite{F}).
We denote by $\overline{G_0(R)}$ and $\overline{A_*(R)}$
the groups modulo numerical equivalence, that is,
\begin{eqnarray*}
\overline{G_0(R)} & = & \frac{G_0(R)}{\{ \alpha \mid \mbox{$\chi_{{\Bbb F}.}(\alpha) = 0$ for any ${\Bbb F}.$} \}} \\
\overline{A_*(R)} & = & \frac{A_*(R)}{\{ \beta \mid \mbox{${\rm ch}({\Bbb F}.)(\beta) = 0$ for any ${\Bbb F}.$} \}}
.
\end{eqnarray*}
It is proven in Proposition~2.4 in \cite{K23} that
numerical equivalence is consistent with dimension of cycles
in $A_*(R)$, so we have
\[
\overline{A_*(R)} = \oplus_{i = 0}^d\overline{A_i(R)} .
\]
The Riemann-Roch map $\tau$ preserves numerical equivalence
as in \cite{K23},
that is, it induces the isomorphism $\overline{\tau}$
that makes the following diagram commutative:
\begin{equation}\label{equation2}
\begin{array}{ccc}
G_0(R)_{\Bbb Q} & \stackrel{\tau}{\longrightarrow} & A_*(R)_{\Bbb Q} \\
\downarrow & & \downarrow \\
\overline{G_0(R)}_{\Bbb Q} & \stackrel{\overline{\tau}}{\longrightarrow} & \overline{A_*(R)}_{\Bbb Q}
\end{array}
\end{equation}
Here $N_K$ denotes $N \otimes_{\Bbb Z} K$.

If $R$ is Cohen-Macaulay,
the Grothendieck group of bounded $R$-free complexes
with support in $\{ {\frak m} \}$ is generated by finite free resolutions of modules of finite length and finite projective dimension (see Proposition~2 in \cite{RS}).
Therefore, in this case,
$\alpha$ in $G_0(R)$ is numerically equivalent to $0$
if and only if $\chi_{{\Bbb F}.}(\alpha) = 0$ for any free resolution ${\Bbb F}.$ of a module with finite length and
finite projective dimension.

\begin{Assumption}\label{assume}
\begin{rm}
Let $R$ be a Noetherian local ring such that,
for each minimal prime ideal ${\frak p}$ of $R$,
there exists a proper generically finite morphism
$Z \rightarrow \spec R/{\frak p}$ such that $Z$ is regular.
\end{rm}
\end{Assumption}

By Hironaka~\cite{Hironaka} and deJong~\cite{dJ},  $R$ satisfies Assumption~\ref{assume}
if $R$ satisfies one of the following two conditions.
\begin{enumerate}
\item
$R$ is an excellent local ring containing ${\Bbb Q}$.
\item
$R$ is essentially of finite type over a field,
${\Bbb Z}$ or a complete discrete valuation ring.
\end{enumerate}

By Theorem~3.1 and  Remark~3.5 in \cite{K23}, both $\overline{G_0(R)}$ and $\overline{A_*(R)}$ are non-zero
finitely generated free abelian group if
$R$ satisfies Assumption~\ref{assume}.

\begin{Example}
\begin{rm}
Let $R$ be a Noetherian local ring satisfying Assumption~\ref{assume}.
\begin{enumerate}
\item
If ${\Bbb K}.$ is the Koszul complex of a system of parameters,
then $\chi_{{\Bbb K}.}([R]) \neq 0$.
Hence, $\overline{G_0(R)} \neq 0$.
\item
If $R$ is a Noetherian local domain with $\dim R \le 2$, then $\mathop{\rm rank} \overline{G_0(R)} = 1$. See Proposition~3.7 in \cite{K23}.
\item
Let $X$ be a smooth projective variety with embedding $X \hookrightarrow {\Bbb P}^n$.
Let $R$ (resp.\ $D$) be the affine cone (resp.\ the very ample divisor) of this embedding.
Then, we have the following commutative diagram:
\[
\begin{array}{ccccc}
G_0(R)_{\Bbb Q} & \stackrel{\sim}{\longrightarrow} & A_*(R)_{\Bbb Q} & 
\stackrel{\sim}{\longleftarrow} & {\rm CH}^\cdot(X)_{\Bbb Q}/ D \cdot {\rm CH}^\cdot(X)_{\Bbb Q} \\
\downarrow & & \downarrow & & \downarrow \\
\overline{G_0(R)}_{\Bbb Q} & \stackrel{\sim}{\longrightarrow} & \overline{A_*(R)}_{\Bbb Q} & 
\stackrel{\phi}{\longleftarrow} & {\rm CH}^\cdot_{num}(X)_{\Bbb Q}/ D \cdot {\rm CH}^\cdot_{num}(X)_{\Bbb Q}
\end{array}
\]
\begin{enumerate}
\item
By the commutativity of this diagram, $\phi$ is a surjection.
Therefore, we have
\begin{equation}\label{uenosiki}
\rank \overline{G_0(R)} \le \dim_{\Bbb Q} {\rm CH}^\cdot_{num}(X)_{\Bbb Q}/ D \cdot {\rm CH}^\cdot_{num}(X)_{\Bbb Q} .
\end{equation}
\item
If ${\rm CH}^\cdot(X)_{\Bbb Q} \simeq {\rm CH}^\cdot_{num}(X)_{\Bbb Q}$, then
we can prove that
 $\phi$ is an isomorphism (\cite{K23}, \cite{RS}).
In this case, the equality holds in (\ref{uenosiki}).
\item
There exists an example such that $\phi$ is not an isomorphism \cite{RS}.

Further, Roberts and Srinivas \cite{RS} proved the following:
Assume that the standard conjecture and Bloch-Beilinson conjecture are true.
Then $\phi$ is an isomorphism if the defining ideal of $R$ is generated by polynomials with coefficients in the algebraic closure of the prime field.
\end{enumerate}
\item
It is conjectured that $\overline{G_0(R)}_{\Bbb Q} \simeq {\Bbb Q}$ if
$R$ is complete intersection isolated singularity with $d$ even.
See Conjecture~3.2 in \cite{DK1}.
\end{enumerate}
\end{rm}
\end{Example}

Let $R$ be a $d$-dimensional Noetherian local ring.
For $\alpha \in G_0(R)_{\Bbb Q}$,
we put \[
\tau(\alpha) = \tau_d(\alpha) + \tau_{d-1}(\alpha) + \cdots + \tau_0(\alpha) ,
\]
where $\tau_i(\alpha) \in A_i(R)_{\Bbb Q}$ (see the diagram~(\ref{equation2})).
Furthermore, assume that $R$ is a normal ring.
Then, we have the {\em determinant map} (or the {\em first Chern class})
$$c_1 : G_0(R) \to A_{d-1}(R)$$
that satisfies
\begin{itemize}
\item
$c_1([R]) = 0$,
\item
$c_1([R/I]) = -c_1([I]) = [\spec R/I]$ for each reflexive ideal $I$,
\item
$c_1([M]) = 0$ if $\dim M \le d-2$.
\end{itemize}
Here, $[\spec R/I]$ denotes
\[
\sum_{P \in {\rm Min}_R(R/I)}\ell_{R_P}((R/I)_P)[\spec R/P] \in  A_{d-1}(R)
\]
for a reflexive ideal $I$ of a Noetherian normal domain $R$.

For an $R$-module $M$, we have
\begin{equation}\label{tau(d-1)}
\tau_{d-1}([M]) = c_1([M]) - \frac{\rank M}{2}K_R
\end{equation}
in $A_{d-1}(R)_{\Bbb Q}$, where $K_R$ is the canonical divisor of $R$,
that is, $K_R = c_1([\omega_R])$ where $\omega_R$ is the canonical module of $R$.
We refer the reader to Lemma~2.9 in \cite{K24} for the above equality.
(Remark that the map $cl$ in \cite{K24} is equal to $-c_1$ in this paper.)

\begin{Lemma}\label{1.3}
Assume that $R$ is a $d$-dimensional Noetherian normal local domain
that satisfies Assumption~\ref{assume}.
Then, there is the map $\overline{c_1}$
that makes the following diagram commutative:
\begin{equation}\label{d1}
\begin{array}{ccc}
G_0(R) & \stackrel{c_1}{\longrightarrow} & A_{d-1}(R) \\
\downarrow & & \downarrow \\
\overline{G_0(R)} & \stackrel{\overline{c_1}}{\longrightarrow} & \overline{A_{d-1}(R)}
\end{array}
\end{equation}
\end{Lemma}

\begin{proof}
First we prove that the map \begin{equation}\label{rank}
{\rm rk} : \overline{G_0(R)} \longrightarrow {\Bbb Z} \end{equation}
taking the rank of a module is well-defined.
Let ${\Bbb K}.$ be the Koszul complex with respect to a parameter ideal $I$.
By the definition of numerical equivalence,
the map \[
\chi_{{\Bbb K}.} : G_0(R) \longrightarrow {\Bbb Z},
\]
taking the alternating sum of the length of homologies
of the complex ${\Bbb K}.$ tensored with
a given $R$-module, induces the map
\[
\chi_{{\Bbb K}.} : \overline{G_0(R)} \longrightarrow {\Bbb Z} .
\]
Then, for an $R$-module $M$,
\[
\chi_{{\Bbb K}.}([M]) = e(I, M) = e(I, R) \cdot \rank M ,
\]
where $e(I, - )$ denotes the Hilbert-Samuel multiplicity with respect to $I$.
Therefore, the map
${\rm rk} : \overline{G_0(R)} \longrightarrow {\Bbb Z}$
taking the rank of a module is well-defined.

Recall that we have a well-defined map
\[
\tau_{d-1} : \overline{G_0(R)}_{\Bbb Q} \longrightarrow \overline{A_{d-1}(R)}_{\Bbb Q}
\]
satisfying (\ref{tau(d-1)}).
Therefore, the map
\[
\overline{c_1} : \overline{G_0(R)}_{\Bbb Q} \longrightarrow
\overline{A_{d-1}(R)}_{\Bbb Q}
\]
which takes the first Chern class
is equal to $\tau_{d-1} + \frac{K_R}{2}{\rm rk}_\QQ$.

Since $\overline{A_{d-1}(R)}$ is torsion-free,
we obtain the well-defined map
\[
\overline{c_1} : \overline{G_0(R)} \longrightarrow
\overline{A_{d-1}(R)}
\]
that makes the diagram~(\ref{d1}) commutative.
\end{proof}

\section{On the boundary of the Cohen Macaulay cone}\label{conesec}
\setcounter{equation}{0}

In this section we prove the main technical results about $C_{CM}(R)$, Theorem \ref{2.1}.

We denote by $\overline{F_{d-1}G_0(R)}$
the kernel of the map  ${\rm rk}$ in (\ref{rank}).
It is easy to check that $\overline{F_{d-1}G_0(R)}$
is generated by cycles $[M]$ with $\dim M < d$.
By the map ${\rm rk}$, we have the decomposition
\begin{equation}\label{decomp}
\overline{G_0(R)} =
\overline{F_{d-1}G_0(R)} \oplus {\Bbb Z}[R] .
\end{equation}

By tensoring  (\ref{decomp}) with the real number field ${\Bbb R}$, we have
\begin{equation}\label{decompR}
\overline{G_0(R)}_{\Bbb R} =
\overline{F_{d-1}G_0(R)}_{\Bbb R}
\oplus {\Bbb R}[R] .
\end{equation}

Let $C_{CM}(R)$ be the Cohen-Macaulay cone defined in Definition~2.4 in \cite{CK}, that is,
\[
C_{CM}(R) = \sum_{M: {\rm MCM}} {\Bbb R}_{\ge 0}[M] \subset \overline{G_0(R)}_{\Bbb R} .
\].

The following result is a main technical result of our paper and plays an important role in the proof of Theorem~\ref{1.2}.

\begin{Theorem}\label{2.1}
Let $(R, {\frak m})$ be a Cohen-Macaulay local domain satisfying Assumption~\ref{assume}.

Then, we have
\[
C_{CM}(R)^- \cap \overline{F_{d-1}G_0(R)}_{\Bbb R} = \{ 0 \} ,
\]
where $C_{CM}(R)^-$ denotes the closure of the Cohen-Macaulay cone
$C_{CM}(R)$ in $\overline{G_0(R)}_{\Bbb R}$
with respect to the classical topology.
\end{Theorem}

\begin{proof}
Let $e_1$, \ldots, $e_s$ be a free basis of $\overline{F_{d-1}G_0(R)}$.
Thinking $[R]$, $e_1$, \ldots, $e_s$ as an  orthonormal basis,
we define a metric on $\overline{G_0(R)}_{\Bbb R}$.
For each vector $v$ in $\overline{G_0(R)}_{\Bbb R}$,
we denote $|| v ||$ the length of the vector $v$.

Assume the contrary.
Let $\alpha$ be a non-zero element in \[
C_{CM}(R)^- \cap \overline{F_{d-1}G_0(R)}_{\Bbb R} .
\]
We may assume that $|| \alpha || = 1$.

By Lemma~2.5 (2) in \cite{CK}, there exists a sequence of
maximal Cohen-Macaulay modules
\[
M_1, M_2, \ldots, M_n, \ldots
\]
such that
\begin{equation}\label{lim}
\lim_{n \rightarrow \infty} \frac{[M_n]}{||[M_n]||} = \alpha
\end{equation}
in $\overline{G_0(R)}_{\Bbb R}$.

Let $x_1$, \ldots, $x_d$ be a system of parameters of $R$.
Since $M_n$ is Cohen-Macaulay,
\begin{equation}\label{sharp}
\ell_R(M_n/{\frak m}M_n) \le \ell_R(M_n/(\underline{x})M_n) =
e((\underline{x}), M_n) = \rank M_n \cdot e((\underline{x}), R) .
\end{equation}
Put $m =e((\underline{x}), R)$ and
$r_n = \rank M_n$ for each $n$.
If $m = 1$, then $R$ is a regular local ring, and therefore,
$\overline{F_{d-1}G_0(R)} = 0$.
Suppose $m \ge 2$. By (\ref{sharp}), we have an exact sequence of the form
\[
0 \longrightarrow N_n \longrightarrow R^{r_nm} \longrightarrow M_n \longrightarrow 0 .
\]
Remark that $N_n$ is a maximal Cohen-Macaulay module.

Set
\[
[M_n] = (\beta_n, r_n[R]) \in \overline{F_{d-1}G_0(R)} \oplus {\Bbb Z}[R] = \overline{G_0(R)}  .
\] Then, we have
\[
[N_n] = r_nm[R] - [M_n]
= (-\beta_n, r_n(m-1)[R]) \in \overline{F_{d-1}G_0(R)} \oplus {\Bbb Z}[R] = \overline{G_0(R)}  .
\] By (\ref{lim}), we have
\begin{eqnarray*}
\lim_{n \rightarrow \infty} \frac{r_n[R]}{||[M_n]||} & = & 0 \\
\lim_{n \rightarrow \infty} \frac{\beta_n}{||[M_n]||} & = & \alpha ,
\end{eqnarray*}
since $\alpha$ is in $\overline{F_{d-1}G_0(R)}_{\Bbb R}$.
Hence, we have
\begin{eqnarray*}
\lim_{n \rightarrow \infty} \frac{r_n(m-1)[R]}{||[M_n]||} & = & 0 \\
\lim_{n \rightarrow \infty} \frac{-\beta_n}{||[M_n]||} & = & -\alpha .
\end{eqnarray*}

On the other hand, it is easy to see
\[
\lim_{n \rightarrow \infty} \frac{||[N_n]||}{||[M_n]||} = 1 .
\]
In fact, since $m \ge 2$,
\[
0 < ||[M_n]|| \le ||[N_n]|| .
\]
Hence,
\[
1 \le \frac{||[N_n]||}{||[M_n]||} \le
\frac{||r_n(m-1)[R]||}{||[M_n]||} + \frac{||\beta_n||}{||[M_n]||}
\longrightarrow ||\alpha|| = 1 .
\]

Then, we have
\begin{eqnarray*}
\lim_{n \rightarrow \infty} \frac{r_n(m-1)[R]}{||[N_n]||} & = & 0 \\
\lim_{n \rightarrow \infty} \frac{-\beta_n}{||[N_n]||} & = & -\alpha .
\end{eqnarray*}
Therefore, \[
\lim_{n \rightarrow \infty} \frac{[N_n]}{||[N_n]||} = - \alpha .
\]
Thus, we have \[
- \alpha \in C_{CM}(R)^- .
\]
It contradicts to Lemma~2.5 (4), (5) in \cite{CK}.
\end{proof}

\begin{Lemma}\label{Lemma2.2}
Let $r$ be a positive integer.
Put
\[
C_{CM}(R)_r^- = \{ \alpha \in  \overline{F_{d-1}G_0(R)}_{\Bbb R} \mid
(\alpha, r[R]) \in C_{CM}(R)^- \} ,
\]
where $C_{CM}(R)^- \subset \overline{G_0(R)}_{\Bbb R} = \overline{F_{d-1}G_0(R)}_{\Bbb R} \oplus {\Bbb R}[R]$.

Then, $C_{CM}(R)_r^-$ is a compact subset of  $\overline{F_{d-1}G_0(R)}_{\Bbb R}$.
\end{Lemma}

\begin{proof}
Tensoring (\ref{rank}) with ${\Bbb R}$, we obtain the map
\[
{\rm rk}_{\Bbb R} : \overline{G_0(R)}_{\Bbb R} \longrightarrow {\Bbb R} .
\]
Then, since
\[
C_{CM}(R)_r^- \simeq C_{CM}(R)^- \cap ({\rm rk}_{\Bbb R})^{-1}(r) ,
\]
$C_{CM}(R)_r^-$ is a closed subset of $ \overline{F_{d-1}G_0(R)}_{\Bbb R} \simeq ({\rm rk}_{\Bbb R})^{-1}(r)$.
(Note that this identification is given by $\alpha \mapsto \alpha  + r[R]$.)

Suppose that $C_{CM}(R)_r^-$ is not bounded.
Then, there exists a sequence of maximal Cohen-Macaulay modules
\[
G_1, \ G_2, \ G_3, \ \ldots
\]
such that, if
\[
G_n = (\alpha_n, s_n[R]) \in \overline{F_{d-1}G_0(R)}_{\Bbb R} \oplus {\Bbb R}[R]
= \overline{G_0(R)}_{\Bbb R},
\]
we have
\begin{equation}\label{limits}
\lim_{n \rightarrow \infty}\frac{r}{s_n} ||\alpha_n|| = \infty .
\end{equation}

Put
\[
S = \{ v \in \overline{F_{d-1}G_0(R)}_{\Bbb R} \mid
||v|| = 1 \} .
\]
Then \[
\frac{\alpha_n}{||\alpha_n||} \in S
\]
if $\alpha_n \neq 0$.
Since $S$ is compact,
$\left\{  \alpha_n/||\alpha_n|| \right\}_n$ contains a subsequence
that converges to a point of $S$, say $\beta$.
Taking a subsequence, we may assume
\[
\lim_{n \rightarrow \infty}\frac{\alpha_n}{||\alpha_n||} = \beta .
\]

Then, 
\[
\lim_{n \rightarrow \infty}\frac{[G_n]}{||[G_n]||} = \beta .
\]
In fact, by (\ref{limits}), we know
\[
\lim_{n \rightarrow \infty}\frac{[R]}{\frac{1}{s_n}||[G_n]||} = 0,
\]
since $0 \le ||\alpha_n||/s_n \le ||[G_n]||/s_n$.
Further, \[
1 = \frac{||[G_n]||}{ ||[G_n]||} \le
\frac{||\alpha_n||}{ ||[G_n]||} +  \frac{||s_n[R]||}{||[G_n]||}
\le 1 + \frac{||[R]||}{\frac{1}{s_n}||[G_n]||} \longrightarrow  1 \ (n \rightarrow \infty) .
\]
Hence \[
\lim_{n \rightarrow \infty}\frac{||\alpha_n||}{||[G_n]||} = 1 .
\]
Therefore,
\[
\lim_{n \rightarrow \infty}\frac{[G_n]}{||[G_n]||} = \lim_{n \rightarrow \infty}\frac{\alpha_n}{||[G_n]||} + \lim_{n \rightarrow \infty}\frac{[R]}{\frac{1}{s_n}||[G_n]||}
= \lim_{n \rightarrow \infty}\frac{\alpha_n}{||\alpha_n||} = \beta
\in S  \subset \overline{F_{d-1}G_0(R)}_{\Bbb R}.
\]

Thus, we have
\[
0 \neq \beta \in C_{CM}(R)^- \cap \overline{F_{d-1}G_0(R)}_{\Bbb R} .
\]
It contradicts  Theorem~\ref{2.1}.
\end{proof}

A  key consequence  is the following.
\begin{Theorem}\label{1.2}
Assume that $R$ is a Cohen-Macaulay local domain
that satisfies Assumption~\ref{assume}.

Then, for any positive integer $r$,
\[
\{ [M] \in \overline{G_0(R)} \mid \mbox{$M$ is a maximal Cohen-Macaulay module of
rank $r$ } \}
\]
is a finite subset of $\overline{G_0(R)}$.
\end{Theorem}

\begin{proof} Assume the contrary: suppose that there exist infinitely many maximal Cohen-Macaulay modules
\[
L_1, L_2, \ldots, L_n, \ldots
\]
such that
\begin{itemize}
\item
$\rank L_n = r$ for all $n > 0$, and
\item
$[L_i] \neq [L_j]$ in $\overline{G_0(R)}$ if $i \neq j$.
\end{itemize}
Here, note \[
[L_i] - r[R] \in  \overline{F_{d-1}G_0(R)} \simeq {\Bbb Z}^s
\]
(for some $s$) for each $i$, and
\[
[L_i] - r[R] \neq [L_j] - r[R]
\]
if $i \neq j$.
Therefore, we have
\[
\lim_{i \rightarrow \infty} ||[L_i] - r[R]|| = \infty.
\]
Now, one can evoke Lemma \ref{Lemma2.2} to finish the proof. 

\end{proof}

The following corollary immediately follows from Theorem~\ref{1.2}
and Lemma~\ref{1.3}.

\begin{Corollary}\label{1.4}
Assume that $R$ is a $d$-dimensional Cohen-Macaulay local normal domain
that satisfies Assumption~\ref{assume}.

Then, for any positive integer $r$,
\[
\{ \overline{c_1}([M]) \in \overline{A_{d-1}(R)} \mid \mbox{$M$ is a maximal Cohen-Macaulay module of
rank $r$ } \}
\]
is a finite subset of $\overline{A_{d-1}(R)}$.
\end{Corollary}

\section{On finiteness of Cohen-Macaulay ideals}\label{sechoch}
\setcounter{equation}{0}

{In this section we discuss  applications of our main technical results in Section \ref{conesec} on the following}:

\begin{Question}
Let $R$ be a local normal domain. If the divisor class group $\Cl(R)$  is finitely generated, then does it contain only finitely many maximal Cohen-Macaulay modules of rank one up to isomorphisms?
\end{Question}

Although this question seems to be  well-known among certain experts,  it is not clear who made it (it was attributed to Hochster in \cite{Kh}). As stated, it needs some adjustments. Here we give a counter example in Example~\ref{elliptic} in dimension two. In this situation, since any reflexive module is a maximal Cohen-Macaulay module, the question merely states that the class group is finitely generated if and only if it is finite. 

\begin{Example}\label{elliptic}
\begin{rm}
Let $E$ be an elliptic curve over $\QQ$ of  rank positive. Then it follows from the Mordell-Weil theorem (see \cite[Theorem 1.9]{CT} for details) that the Picard group of $E$ is finitely generated. But since the group of $\QQ$-rational points in $E$ has  rank positive, the rank of  $\Pic(E)$ ($= \Cl(E)$) is at least $2$. Now take the cone over $E$ and let $R$ be the local ring at the vertex. Then $\Cl(R) =\Cl(E)/\ZZ H $ is finitely generated with positive rank.  By adjoining new variables to $R$ one can get examples in all higher dimensions. 
\end{rm}
\end{Example}

It is natural to speculate the following:

\begin{Conjecture}\label{finCM}
Let $R$ be a local normal domain over an algebraically closed field. If $\Cl(R)$  is finitely generated, then  it contains only finitely many maximal Cohen-Macaulay modules of rank one up to isomorphisms.
\end{Conjecture}

\begin{Conjecture}
Let $R$ be a local normal domain. If $\Cl(R^{sh})$ (the class group of strict henselization) is finitely generated , then
there exist only finitely many maximal Cohen-Macaulay modules of rank one up to isomorphisms.
\end{Conjecture}

\begin{Conjecture}
Let $R$ be an {excellent} local normal domain with isolated singularity. If the dimension of $R$ is bigger than or equal to $3$, then
there exist only finitely many maximal Cohen-Macaulay modules of rank one up to isomorphisms.
\end{Conjecture}

By the commutativity of diagram~(\ref{d1}),
the following corollary easily follows from Corollary~\ref{1.4}.

\begin{Corollary}\label{1.5}
Assume that $R$ is a $d$-dimensional Cohen-Macaulay local normal domain
that satisfies Assumption~\ref{assume}.
Assume that the kernel of the natural map
\begin{equation}\label{star}
A_{d-1}(R) \longrightarrow \overline{A_{d-1}(R)}
\end{equation}
is a finite group.

Then, for any positive integer $r$,
\[
\{ c_1([M]) \in A_{d-1}(R) \mid \mbox{$M$ is a maximal Cohen-Macaulay module of
rank $r$ } \}
\]
is a finite subset of $A_{d-1}(R)$.

In particular, $R$ has only finitely many maximal
Cohen-Macaulay modules of rank one up to isomorphism.
\end{Corollary}

By Corollary~\ref{1.5}, if the kernel of the map (\ref{star}) is a finite set,
there exist only finitely many maximal Cohen-Macaulay modules of rank one
under a mild condition.
As in the following theorem due to Danilov (Lemma~4 in \cite{Danilov1} and Theorem~1, Corollary~1 in \cite{Danilov2}),
$A_{d-1}(R)$ is finitely generated for most of Cohen-Macaulay local normal domain of dimension at least three.
Note that, if $A_{d-1}(R)$ is finitely generated and if $A_{d-1}(R)_\QQ \rightarrow \overline{A_{d-1}(R)}_\QQ$ is an isomorphism,
then the kernel of (\ref{star}) is a finite set.

\begin{Theorem} (Danilov)
Let $R$ be a equi-characteristic excellent local normal domain with isolated singularity.
Assume that $R$ satisfies one of the following two conditions:
\begin{itemize}
\item[a)]
$R$ is essentially of finite type over a field of characteristic zero.
\item[b)]
There exists a maximal primary ideal $I$ of $R$ such that
the blow-up at $I$ is a regular scheme.
\end{itemize}

If $\depth R \ge 3$, then ${\rm Cl}(R)$ is finitely generated.
\end{Theorem}

In the case of dimension $2$,
there exist examples of isolated hypersurface singularity
that has infinitely many maximal Cohen-Macaulay modules of rank one
(see Example~\ref{elliptic}).

{Our next corollary confirms the Conjectures considered in this section for most isolated hypersurface singularities of dimension at least $3$. 

\begin{Corollary}\label{1.6}
Let $R$ be a $3$-dimensional isolated hypersurface singularity
with desingularization $f : X \rightarrow \spec R$ such that
$X \setminus f^{-1}({\frak m}) \simeq \spec R \setminus \{ {\frak m} \}$.
Then $R$ has only finitely many maximal
Cohen-Macaulay modules of rank one up to isomorphism.
\end{Corollary}

\begin{proof}
This follows immediately from Corollary~\ref{1.5} and Proposition \ref{1.6}.
\end{proof}}

\begin{Example}\label{gradedcase}
\begin{rm}
Suppose that $A$
is a positively graded ring over a field $k$, that is,
$A = \oplus_{n \geq 0}A_n$.
Put $R = A_{A_+}$ and $X = \proj A$.
We assume that $X$ is smooth over $k$.

{Further, assume that $\depth A\geq 3 $,
the characteristic of $k$ is zero and $k$ is algebraically closed.}

In this case, since $H^1(X, {\mathcal O}_X) = 0$, ${\rm Pic}(X)$ is finitely generated.
Therefore, ${\rm Cl}(R)$ is also a finitely generated abelian group.
Then, there exists only finitely many maximal Cohen-Macaulay modules
of rank one up to isomorphism.
It is essentially written in Karroum~\cite{Kh}, Theorem 6.11. {We thank H. Flenner for explaining this result to us.} 
\end{rm}
\end{Example}

\vspace{2mm}

\section{On the kernel of the map (\ref{star})}\label{4.1}

{In this section we study the question when the kernel of the map (\ref{star}):  $$A_{d-1}(R) \longrightarrow \overline{A_{d-1}(R)}$$ to be a finite set. This turns out to be a rather deep question, even in the graded case. However, our next Proposition (and Remark) establish it for most isolated hypersurfaces singularities of dimension at least $3$. } 

\begin{Proposition}\label{1.6}
Let $R$ be a $3$-dimensional isolated hypersurface singularity
with desingularization $f : X \rightarrow \spec R$ such that
$X \setminus f^{-1}({\frak m}) \simeq \spec R \setminus \{ {\frak m} \}$.
Then, the natural map
\[
A_{2}(R) \longrightarrow \overline{A_{2}(R)}
\]
is an isomorphism.
\end{Proposition}

\begin{proof}
Let $I$ be a reflexive ideal.
Assume that $c_1([I])$ is numerically equivalent to zero
in $A_{2}(R)$.
We shall prove $c_1([I]) = 0$ in $A_{2}(R)$.

Put
\[
\tau([I]) =\tau_3([I]) + \tau_2([I]) + \tau_1([I]) + \tau_0([I]) ,
\]
where $\tau_i([I]) \in A_i(R)_{\Bbb Q}$ for $i = 0, 1, 2, 3$.
By the top term property and that $R$ is a complete intersection,
\[
\tau_3([I]) = [\spec R] = \tau([R]) \]
in $A_*(R)_{\Bbb Q}$.
By the equality~(\ref{tau(d-1)}), we have
\[
\tau_2([I]) = c_1([I])
\]
in $A_2(R)_{\Bbb Q}$, where it is numerically equivalent to zero
by our assumption.
By Proposition~3.7 in \cite{K23},
$\tau_1([I])$ and $\tau_0([I])$ are numerically equivalent to zero in the Chow group.
Therefore, $\tau([I])$ is numerically equivalent to $\tau([R])$
in the Chow group.
Then, $[I]$ is numerically equivalent to $[R]$ in the Grothendieck group.

Let $\theta^R$ be the Hochster's theta function (see \cite{DK1}).
Then we have
\[
\theta^R(I, I) =  \theta^R(R, I) = 0
\]
by Corollary~6.3 (1) in \cite{DK1}.
Then, by Corollary~7.9 in \cite{DK1}, $c_1([I]) = 0$ in $A_2(R)$.
\end{proof}

\begin{Remark}

If $R$ is a complete intersection with isolated singularity
of dimension at least four,
then $R$ is a unique factorization domain
(Lemma 3.16, 3.17 in \cite{SGA6} or \cite{CL}).
Therefore, the map (\ref{star}) is automatically an isomorphism. 
\end{Remark}

In the rest of this section, suppose that $A$
is a standard graded ring over a field $k$, that is,
$A = \oplus_{n \geq 0}A_n = k[A_1]$.
Put $R = A_{A_+}$ and $X = \proj A$.
We assume that $X$ is smooth over $k$.
Put $d = \dim A$ and $n = \dim X > 0$.
Of course, $d = n+1$.
Let ${\rm CH}^i(X)$ (resp.\  ${\rm CH}^i_{num}(X)$)
be the Chow group of $X$ (resp.\ the Chow group of $X$ modulo numerical equivalence) of codimension $i$.
Put $h = c_1({\mathcal O}_X(1)) \in {\rm CH}^1(X)$.
By (7.5) in \cite{K23}, we have the induced map
\begin{equation}\label{themapf}
f : {\rm CH}^1_{num}(X)_{\Bbb Q}/h{\rm CH}^0_{num}(X)_{\Bbb Q}
\longrightarrow \overline{A_{d-1}(R)}_{\Bbb Q} .
\end{equation}
Consider the following natural map:
\begin{equation}\label{mapg}
g : {\rm ker}\left(
{\rm CH}^{n-1}(X)_{\Bbb Q} \stackrel{h}{\rightarrow}
{\rm CH}^{n}(X)_{\Bbb Q}
\right)
\longrightarrow
{\rm ker}\left(
{\rm CH}^{n-1}_{num}(X)_{\Bbb Q} \stackrel{h}{\rightarrow}
{\rm CH}^{n}_{num}(X)_{\Bbb Q}
\right) , \end{equation}
where the map $h$ means the multiplication by $h$.

Here, we obtain the following lemma which is essentially due to Roberts-Srinivas~\cite{RS}.

\begin{Lemma}\label{LemmaSection8}
$\dim_{\Bbb Q} {\rm ker}(f) = \dim_{\Bbb Q} {\rm coker}(g)$.
\end{Lemma}

\begin{proof}
If $\dim X = 1$, then both of ${\rm ker}(f)$ and ${\rm coker}(g)$ are zero.

Suppose $\dim X \ge 2$.
Consider the following diagram:
\[
\begin{array}{ccccc}
{\rm CH}^1(X)\subq & \longrightarrow & {\rm CH}^1(X)\subq/h{\rm CH}^0(X)\subq & \longrightarrow & {\rm CH}^1_{num}(X)\subq/h{\rm CH}^0_{num}(X)\subq \\
& & \parallel & & {\scriptstyle f}\downarrow \phantom{\scriptstyle f}
\\
& & A_{d-1}(R)\subq & \stackrel{p}{\longrightarrow} & \overline{A_{d-1}(R)}\subq
\end{array}
\]
Here $p$ is the natural surjection.
Let $K^{\frak m}_0(R)\subq$ be the Grothendieck group of bounded finite free $R$-complexes with homologies of finite length.
We define \[
\varphi : K^{\frak m}_0(R)\subq \longrightarrow {\rm CH}^{\bullet}(X)\subq
\]
as in (7.2) in \cite{K23}.
Let $\overline{\beta}$ (in $A_{d-1}(R)\subq$) be the image of
$\beta$ (in ${\rm CH}^1(X)\subq$).
By (7.6) in \cite{K23}, for any $\alpha \in K^{\frak m}_0(R)\subq$, we have
\[
{\rm ch}(\alpha) \left( \overline{\beta} \right)
= \pi_*(\varphi(\alpha) \cdot \beta) ,
\]
where $\pi : X \rightarrow \spec k$ is the structure map,
and $\varphi(\alpha) \cdot \beta$ is the intersection product
in the Chow ring of $X$.
Thus, we know the following:
\begin{equation}\label{equivalence}
\begin{array}{cl}
& \mbox{$p(\overline{\beta}) = 0$ in $\overline{A_{d-1}(R)}\subq$} \\
\Longleftrightarrow &
\mbox{$\forall \alpha \in K^{\frak m}_0(R)\subq$, ${\rm ch}(\alpha) \left( \overline{\beta} \right) = 0$} \\
\Longleftrightarrow &
\mbox{$\forall \alpha \in K^{\frak m}_0(R)\subq$, $\pi_*(\varphi(\alpha) \cdot \beta) = 0$} \\
\Longleftrightarrow &
\mbox{$\forall \alpha' \in {\rm ker}\left(
{\rm CH}^{n-1}(X)_{\Bbb Q} \stackrel{h}{\rightarrow}
{\rm CH}^{n}(X)_{\Bbb Q}
\right)$, $\pi_*(\alpha' \cdot \beta) = 0$}
\end{array}
\end{equation}
by the exact sequence (7.2) in \cite{K23}.

Consider the perfect pairing
\begin{equation}\label{prod1}
{\rm CH}^1_{num}(X)\subq \times {\rm CH}^{n-1}_{num}(X)\subq
\longrightarrow {\rm CH}^{n}_{num}(X)\subq = {\Bbb Q}
\end{equation}
induced by the intersection product.
We define ${\Bbb Q}$-vector subspaces as follows:
\begin{eqnarray*}
V & = & \left\{
\gamma \in {\rm CH}^1_{num}(X)\subq \mid
\mbox{$h^{n-1}\gamma = 0$ in ${\rm CH}^n_{num}(X)\subq$}
\right\} \subset {\rm CH}^1_{num}(X)\subq \\
U & = & \left\{
\delta \in {\rm CH}^{n-1}_{num}(X)\subq \mid
\mbox{$h\delta = 0$ in ${\rm CH}^n_{num}(X)\subq$}
\right\} \subset {\rm CH}^{n-1}_{num}(X)\subq
\end{eqnarray*}
Then, it is easy to see
\begin{eqnarray*}
{\rm CH}^1_{num}(X)\subq & = & h{\rm CH}^0_{num}(X)\subq
\oplus V \\
{\rm CH}^{n-1}_{num}(X)\subq & = & h^{n-1}{\rm CH}^0_{num}(X)\subq
\oplus U
\end{eqnarray*}
The intersection pairing (\ref{prod1})
induces the following perfect pairing:
\begin{equation}\label{prod2}
V \times U
\longrightarrow {\rm CH}^{n}_{num}(X)\subq = {\Bbb Q}
\end{equation}
Here remark that $\dim_{\Bbb Q} V = \dim_{\Bbb Q} U = \dim_{\Bbb Q} {\rm CH}^1_{num}(X)\subq - 1$.
Put $W = {\rm Im}(g) \subset U$,
where $g$ is the map in (\ref{mapg}).
Take a ${\Bbb Q}$-vector subspace $W_1$ of $U$ such that
\[
U = W \oplus W_1 .
\]
We define a ${\Bbb Q}$-vector subspace of $V$ as follows:
\[
V_1 =  \left\{
v \in V \mid
\forall w \in W, \ v\cdot w = 0
\right\} \subset V
\]
The intersection pairing (\ref{prod2})
induces the following perfect pairing:
\[
V_1 \times W_1 \longrightarrow {\rm CH}^{n}_{num}(X)\subq = {\Bbb Q}
\]
Then,
\[
\dim_{\Bbb Q} {\rm coker}(g) = \dim_{\Bbb Q} W_1
= \dim_{\Bbb Q} V_1 .
\]
Note that the composite map of
\[
V \hookrightarrow {\rm CH}^1_{num}(X)\subq
\longrightarrow {\rm CH}^1_{num}(X)\subq/h{\rm CH}^0_{num}(X)\subq
\]
is an isomorphism.
We identify $V$ with ${\rm CH}^1_{num}(X)\subq/h{\rm CH}^0_{num}(X)\subq$.
Then, we have an exact sequence
\[
0 \longrightarrow V_1 \longrightarrow {\rm CH}^1_{num}(X)\subq/h{\rm CH}^0_{num}(X)\subq
\stackrel{f}{\longrightarrow} \overline{A_{d-1}(A)}\subq \longrightarrow 0 \]
because, for $v \in V$, $f(v)$ is equal to $0$ if and only if
$w \cdot v = 0$ for any $w \in W$ by (\ref{equivalence}).
Therefore, $\dim_{\Bbb Q} {\rm ker}(f) = \dim_{\Bbb Q} {\rm coker}(g)$.
\end{proof}

\begin{Proposition}\label{Prop3.8}
Let $A$ be a standard graded Cohen-Macaulay ring over a field $k$
of characteristic zero.
Assume that $X = \proj A$ is smooth over $k$.
Put $R = A_{A_+}$. Let $d = \dim X + 1 \ge 3$.

Then, the kernel of (\ref{star}) is a finite set
if and only if the map $g$ in (\ref{mapg}) is surjective. 
In particular, if ${\rm CH}^{\dim X}(X)_{\Bbb Q} \simeq {\Bbb Q}$,
then the kernel of (\ref{star}) is a finite set.
\end{Proposition}

\begin{proof}
The natural map \[
{\rm CH}^1(X)_{\Bbb Q} \longrightarrow {\rm CH}_{num}^1(X)_{\Bbb Q}
\]
is an isomorphism and ${\rm CH}^1(X)$ is finitely generated
since $H^1(X, {\mathcal O}_X) = 0$.
Then, by Lemma~\ref{LemmaSection8},
$g$ is surjective if and only if $A_{d-1}(R)_{\QQ}$ is isomorphic to $\overline{A_{d-1}(R)}_{\QQ}$.
\end{proof}

Note that, in the case of $\dim X = 2$ under the situation in Proposition~\ref{Prop3.8},  
$g$ is surjective if and only if the image of the map $${\rm CH}^1(X)_{\Bbb Q} \stackrel{h}{\rightarrow}
{\rm CH}^2(X)_{\Bbb Q}$$ is isomorphic to ${\Bbb Q}$,
that is, for $\alpha \in {\rm CH}^1(X)$, if the degree of $h\alpha$ is zero,
then $h\alpha$ is a torsion in ${\rm CH}^2(X)$.

Some varieties (e.g., Fano variety, toric variety, etc) satisfies ${\rm CH}^{\dim X}(X)_{\Bbb Q} = {\Bbb Q}$. {In fact we have: 

\begin{Proposition}
Let $A$ be a standard graded Cohen-Macaulay ring over a field $k$
of characteristic zero.
Assume that $X = \proj A$ is smooth over $k$.
Put $R = A_{A_+}$. Consider the following
\begin{enumerate}
\item $X$ is Fano.
\item ${\rm CH}^{\dim X}(X)_{\Bbb Q} = {\Bbb Q}$.
\item $R$ has rational singularity. 
\end{enumerate}

Then $(1) \implies (2) \implies (3)$. If $R$ is Gorenstein, all the three conditions are equivalent. 

\end{Proposition}

\begin{proof}
It is well known that Fano varieties are rationally connected, so $(1) \implies (2)$. If $X$ satisfies ${\rm CH}^{\dim X}(X)_{\Bbb Q} = {\Bbb Q}$
then $R$ has only a rational singularity (use Lemma~3.9 in \cite{DITV}). Finally, if $R$ is Gorenstein and has rational singularity, then the graded canonical module of $A$ is generated in positive degree, so $X$ is Fano by definition.  
\end{proof}}

\begin{Proposition}\label{1.7}
Let $A$ be a standard graded Cohen-Macaulay ring of dimension $d \ge 3$.
Assume that defining equations can be chosen to be polynomials
with coefficients algebraic over the prime field.
Let $R$ be the affine cone of $\proj A$.
Assume that $\proj A$ is smooth over the field $A_0$.

If some conjectures (the standard conjecture and the Bloch-Beilinson conjecture) on algebraic cycles are true, then the kernel of the natural map
\[
A_{d-1}(R) \longrightarrow \overline{A_{d-1}(R)}
\]
is a finite group.
\end{Proposition}

We can prove the above proposition in the same way as in Section~5 in Roberts-Srinivas~\cite{RS}.

\begin{Example}
\begin{rm}
The map $f$ in (\ref{themapf}) is not necessary isomorphism as in Section~5 in Roberts-Srinivas~\cite{RS}.
We give an example here.

Let $A$ be a standard graded Cohen-Macaulay domain over $\CC$ of dimension $3$ with an isolated singularity.
Assume that $A$ is a unique factorization domain and
$H^2(X, {\mathcal O}_X) \neq 0$, where $X = \proj A$.
For example, $\CC[x,y,z,w]/(f(x,y,z,w))$ satisfies these assumptions for any general homogeneous form $f(x,y,z,w)$ of degree $\ge 4$
(by Noether-Lefschetz theorem).

Let $H$ be a divisor corresponding to ${\mathcal O}_X(1)$.
Then, $H^2$ is not zero in ${\rm CH}^2(X)_\QQ$.
By Mumford's infinite dimensionality theorem (e.g.\ Lemma~3.9 in \cite{DITV})
for $0$-cycles, ${\rm CH}^2(X)_\QQ$ is of dimension infinite.
Let $r$ be a positive integer.
Let $p_1$, \ldots, $p_r$ be closed points of $X$ such that $H^2$, $p_1$, \ldots, $p_r$ (in ${\rm CH}^2(X)_\QQ$) are linearly independent over $\QQ$.
Let $\pi : Q \rightarrow X$ be the blow-up at $\{ p_1, \ldots, p_r \}$.
Put $E_i = \pi^{-1}(p_i)$ for $i = 1, \ldots, r$.
Choose positive integers $a$, $b_1$, \ldots, $b_r$ such that
\[
a\pi^*H - b_1E_1 - \cdots - b_rE_r
\]
is a very ample, projectively normal divisor on $Q$.
We denote it by $D$.
Then,
\[
{\rm CH}^1(Q) \simeq {\rm CH}_{num}^1(Q) = \ZZ \pi^*H + \ZZ E_1 + \cdots + \ZZ E_r \simeq \ZZ^{r+1} .
\]

By the construction, one can prove that
\[
{\rm CH}^1(Q)_{\QQ} \stackrel{D}{\longrightarrow} {\rm CH}^2(Q)_{\QQ}
\]
is injective.
Since the kernel of
\[
{\rm CH}_{num}^1(Q)_{\QQ} \stackrel{D}{\longrightarrow} {\rm CH}_{num}^2(Q)_{\QQ}
= {\QQ}
\]
is ${\QQ}^r$, the cokernel of the map $g$ in (\ref{mapg}) is ${\QQ}^r$.
Therefore, by Lemma~\ref{LemmaSection8}, the kernel of the map $f$
in (\ref{themapf}) is  ${\QQ}^r$.

However, remember that 
\[
\oplus_{n \ge 0} H^0(Q, {\mathcal O}_Q(nD))
\]
has only finitely many maximal Cohen-Macaulay modules of rank one
by Example~\ref{gradedcase}.
\end{rm}
\end{Example}

\section{Some explicit examples of the Cohen-Macaulay cones}\label{CMcones}

In this section, we compute Cohen-Macaulay cones for certain hypersurfaces. One of the main tools we use 
is Kn\"orrer periodicity \cite{K}.

We define $C'_{CM}(R)$ to be the cone spanned by maximal Cohan-Macaulay modules in $G_0(R)_{\Bbb R}$,
that is,
\[
C'_{CM}(R) = 
\sum_{M: {\rm MCM}} {\Bbb R}_{\ge 0}[M] \subset G_0(R)_{\Bbb R}.
\]

\begin{Theorem}\label{Th6.1}
Let $k$ be a field.
Put $R = k[[x_1, \ldots, x_n]]/(f)$.
Suppose $0 \neq f=f_1^{a_1}f_2^{a_2}\cdots f_m^{a_m} \in k[[x_1, \ldots, x_n]]$ with each $f_i$ irreducible, and $f_1$, \ldots, $f_m$ are pairwise coprime.
Let $R^{\scriptscriptstyle \# \#} = k[[x_1, \ldots, x_n,\xi,\eta]]/(\xi \eta + f)$.
For  $1 \le j_1 < j_2 < \cdots < j_t \le n$,
let $I_{j_1j_2\cdots j_t}$ denote the ideal $(\eta, f_{j_1}^{a_{j_1}}f_{j_2}^{a_{j_2}}\cdots f_{j_t}^{a_{j_t}})$. 

We assume that, if $N$ is an $R$-module with $\dim N < \dim R$, then
$[N] = 0$ in $G_0(R)_{\Bbb Q}$.
(If $n = 2$, it is always satisfied.)

\begin{enumerate}
\item
Suppose $m = 1$.
Then, we have
\[
G_0(R)_{\RR} = {\RR}[R/(f_1)] \supset C'_{CM}(R) =  {\RR}_{\ge 0}[R/(f_1)]
\]
and
\[
G_0(R^{\scriptscriptstyle \# \#})_{\RR} = {\RR}[R^{\scriptscriptstyle \# \#}] \supset C'_{CM}(R^{\scriptscriptstyle \# \#}) =  {\RR}_{\ge 0}[R^{\scriptscriptstyle \# \#}] .
\]
\item
Suppose $m \ge 2$.
Then, we have
\[
G_0(R)_{\RR} = \oplus_{i = 1}^m{\RR}[R/(f_i)] \supset C'_{CM}(R) = 
\sum_{i = 1}^m {\RR}_{\ge 0}[R/(f_i)]
\]
and
\[
G_0(R^{\scriptscriptstyle \# \#})_{\RR} = \oplus_{i = 1}^m{\RR}[I_i] .
\]
Furthermore, $C'_{CM}(R^{\scriptscriptstyle \# \#})$ is minimally spanned by 
by 
\[
\{ [I_{j_1j_2\cdots j_t}] \mid \emptyset \neq \{ j_1, \ldots, j_t \} \subsetneq \{1,~\ldots, n \} \}.
\]
\item
Suppose $n = 2$.
Assume that there exists a resolution of singularity
$\pi : X \rightarrow \spec R^{\scriptscriptstyle \# \#}$ such that
$X \setminus \pi^{-1}(V((x_1, x_2, \xi, \eta)R^{\scriptscriptstyle \# \#}))
\rightarrow  \spec R^{\scriptscriptstyle \# \#}
\setminus V((x_1, x_2, \xi, \eta)R^{\scriptscriptstyle \# \#})$ is an isomorphism.
Then, the natural map  $G_0(R^{\scriptscriptstyle \# \#})_{\QQ} 
\rightarrow \overline{G_0(R^{\scriptscriptstyle \# \#})}_{\QQ}$
is an isomorphism.
In particular, the cone $C'_{CM}(R^{\scriptscriptstyle \# \#})$ coincides with 
 the Cohen-Macaulay cone $C_{CM}(R^{\scriptscriptstyle \# \#})$.
\end{enumerate}
\end{Theorem}

We refer the reader to \cite{Y} for the terminologies and the basic theory on maximal Cohen-Macaulay modules.

Using Kn\"orrer periodicity \cite{K}, we have the category equivalence
\[
\Omega : \underline{\frak C}(R) \rightarrow \underline{\frak C}(R^{\scriptscriptstyle \# \#}) ,
\]
where $\underline{\frak C}(R)$ (resp.\ $\underline{\frak C}(R^{\scriptscriptstyle \# \#})$) denotes the stable category of maximal Cohen-Macaulay $R$-modules (resp.\ $R^{\scriptscriptstyle \# \#}$-modules).
In order to prove Theorem~\ref{Th6.1}, we need the following claim:

\begin{Claim}\label{claim6-2}
Under the same situation as in Theorem~\ref{Th6.1},
the functor ${\Omega}$ induces the natural isomorphism
\begin{equation}\label{6-1}
G_0(R)/{\Bbb Z}[R] \simeq G_0(R^{\scriptscriptstyle \# \#})/{\Bbb Z}[R^{\scriptscriptstyle \# \#}] .
\end{equation}
\end{Claim}

\begin{proof}
By Theorem~4.4.1 in \cite{B}, $\underline{\frak C}(R)$ has a structure of a
triangulated category since $R$ is a Gorestein ring.
We can define the Grothendieck group $K_0(\underline{\frak C}(R))$ 
as a triangulated category.
By 4.9 in \cite{B}, we have an isomorphism
\[
K_0(\underline{\frak C}(R)) \simeq G_0(R)/\ZZ [R] .
\]
Since $R^{\scriptscriptstyle \# \#}$ is Gorenstein, we also obtain
\[
K_0(\underline{\frak C}(R^{\scriptscriptstyle \# \#})) \simeq G_0(R^{\scriptscriptstyle \# \#})/\ZZ [R^{\scriptscriptstyle \# \#}] .
\]
Since Kn\"orrer periodicity $\Omega : \underline{\frak C}(R) \rightarrow \underline{\frak C}(R^{\scriptscriptstyle \# \#})$ is a category equivalence
as triangulated categories, we have
\[
K_0(\underline{\frak C}(R)) \simeq K_0(\underline{\frak C}(R^{\scriptscriptstyle \# \#})) .\]
\end{proof}

Now, we start to prove Theorem~\ref{Th6.1}.

It is easy to prove (1). 
We omit a proof.

We shall prove (2).
Since $[N] = 0$ in $G_0(R)_{\Bbb Q}$ for any $R$-module 
$N$ with $\dim N < \dim R$, it is easy to see
\[
G_0(R)_{\RR} = \oplus_{i = 1}^m{\RR}[R/(f_i)]
 = \oplus_{i = 1}^m{\RR}[R/(f_i^{a_i})]
\]
and
\[
C'_{CM}(R) = \sum_{i = 1}^m {\RR}_{\ge 0}[R/(f_i)]
 = \sum_{i = 1}^m{\RR}_{\ge 0}[R/(f_i^{a_i})] .
\]
Here remark that $[R/(f_i^{a_i})] = a_i [R/(f_i)]$.

By Kn\"orrer periodicity, we have a bijection between the set of isomorphism classes
of indecomposable maximal Cohan-Macaulay $R$-modules
and that of $R^{\scriptscriptstyle \# \#}$.
For an  indecomposable maximal Cohan-Macaulay $R$-module $N$,
we denote by  $\Omega(N)$ the corresponding 
 indecomposable maximal Cohan-Macaulay $R^{\scriptscriptstyle \# \#}$-module.
Then, by definition, we have $\Omega(R/(f_{j_1}^{a_{j_1}}\cdots f_{j_t}^{a_{j_t}})) = 
I_{j_1j_2 \cdots j_t}$.

Since $[R] = \sum_{i = 1}^m[R/(f_i^{a_i})]$,
we have
\[
G_0(R)_{\RR}/{\RR}[R] 
 = \frac{\oplus_{i = 1}^m{\RR}[R/(f_i^{a_i})]}{{\RR}(\sum_{i = 1}^m[R/(f_i^{a_i})])} .
\]
By Claim~\ref{claim6-2}, we have
\begin{equation}\label{20}
G_0(R^{\scriptscriptstyle \# \#})_{\RR}/{\RR}[R^{\scriptscriptstyle \# \#}] 
= \frac{\oplus_{i = 1}^m{\RR}[I_i]}{{\RR}(\sum_{i = 1}^m[I_i])} .
\end{equation}
Since 
\[
[R/I_{j_1j_2 \cdots j_t}] = \sum_{i = 1}^t[R/I_{j_i}] ,
\]
we have
\[
[I_{j_1j_2 \cdots j_t}] + (t-1)[R^{\scriptscriptstyle \# \#}] = 
\sum_{i = 1}^t[I_{j_i}] .
\]
In particular, we have
\begin{equation}\label{21}
m[R^{\scriptscriptstyle \# \#}] = \sum_{i = 1}^m[I_i] .
\end{equation}
By (\ref{20}) and (\ref{21}), we have
\[
G_0(R^{\scriptscriptstyle \# \#})_{\RR} = \oplus_{i = 1}^m{\RR}[I_i] .
\]

Let $C"$ be the cone in $G_0(R)_{\RR}$
spanned by 
by 
\[
\{ [I_{j_1j_2\cdots j_t}] \mid \emptyset \neq \{ j_1, \ldots, j_t \} \subsetneq \{1,~\ldots, n \} \}.
\]
We shall prove $C" = C'_{CM}(R^{\scriptscriptstyle \# \#})$.
By definition, $C" \subset C'_{CM}(R^{\scriptscriptstyle \# \#})$.
It is sufficient to show that $[M]$ is in $C"$ for any 
indecomposable maximal Cohen-Macaulay $R^{\scriptscriptstyle \# \#}$-module
$M$.
By (\ref{21}), $[R^{\scriptscriptstyle \# \#}] \in C"$.

Let $M$ be an indecomposable maximal Cohen-Macaulay $R^{\scriptscriptstyle \# \#}$-module with $M \not\simeq R^{\scriptscriptstyle \# \#}$.

Suppose that $[M] = \sum_{i=1}^{n} x_i[I_i]$ 
in $G_0(R^{\scriptscriptstyle \# \#})_{\RR}$ where $x_i$'s are rational numbers. 
Without loss of generalities we may assume that  $x_1 \leq x_2\leq \dots \leq x_n$. We shall rewrite the above equation as follows:
\begin{eqnarray*}
& & [M] \\
& = & x_1([I_1]+\dots + [I_n])+ (x_2-x_1)([I_2]+\dots+ [I_n]) +\cdots + (x_n-x_{n-1})[I_n] \\
 &  = &  x_1([I_{12\cdots n}] +(n-1)[R^{\scriptscriptstyle \# \#}]) + (x_2-x_1)([I_{2\cdots n}] +(n-2)[R^{\scriptscriptstyle \# \#}]) + \cdots + (x_n-x_{n-1})[I_n]\\
       &  = &  (\sum_{i = 1}^n{x_i} +(x_1-x_n))[R^{\scriptscriptstyle \# \#}] + (x_2-x_1)[I_{2\cdots n}] + \dots  + (x_n-x_{n-1})[I_n]\\
       &  = &  (\rank(M) +(x_1-x_n))[R^{\scriptscriptstyle \# \#}] + (x_2-x_1)[I_{2\cdots n}] + \cdots  + (x_n-x_{n-1})[I_n]
\end{eqnarray*}
As $x_{i+1} -x_i\geq 0$ for all $i$ it remains to show that $\rank(M) \geq x_n-x_1$. 
Using Kn\"orrer periodicity \cite{K}, 
we can assume our module $M$ has a form $M= \Omega(N)$ 
for some indecomposable maximal Cohen-Macaulay $R$-module $N$ 
with $N \not\simeq R$.

We claim that $\mu(M) = 2\rank(M)$.
Since $R^{\scriptscriptstyle \# \#}$ has multiplicity $2$ and
$M$ has no free summand,
we have $$2\rank(M) \geq \mu(M) = \rank(M) + \rank({\rm Syz}^1_{R^{\scriptscriptstyle \# \#}} M),$$ which implies $\rank(M) \geq \rank({\rm Syz}^1_{R^{\scriptscriptstyle \# \#}} M)$. However, as ${\rm Syz}^2_{R^{\scriptscriptstyle \# \#}} M = M$, we must have equality.   

Since the functor $\Omega$ doubles the number of generators
for a maximal Cohen-Macaulay $R$-module
with no free summand, 
we have $\rank(M) = \mu(N)$.  
Suppose $[N] = \sum_i y_i[R/(f_i^{a_i})]$ in $G_0(R)_{\Bbb R}$, 
here 
\begin{equation}\label{22}
y_i a_i =\length_{R_{(f_i)}}N_{(f_i)} . 
\end{equation}
Then we have 
\[
\sum_{i=1}^{n} x_i[I_i] = [M] = [\Omega(N)] =
\sum_i y_i[\Omega(R/(f_i^{a_i}))] = \sum_i y_i[I_i] \ \ 
\mbox{in $G_0(R^{\scriptscriptstyle \# \#})/{\RR}[R^{\scriptscriptstyle \# \#}]$} .
\]
By (\ref{20}),  we have
\[
x_1 - y_1 = x_2 - y_2 = \cdots = x_m - y_m .
\]
In particular, we have $x_m-x_1 = y_m-y_1$. 
So to finish the proof we need to show that $\mu(N)\geq y_m-y_1$. 

As we have a surjection $R^{\mu(N)} \to N$, by localizing at $(f_m)$ 
we have a surjection
\[
R_{(f_m)}^{\mu(N)} \to N_{(f_m)} .
\]
By counting the length, we have
\[
\mu(N) a_m \ge y_m a_m 
\]
by (\ref{22}).
Since $y_1 \ge 0$ by (\ref{22}),
we get $\mu(N)\geq y_m\geq y_m-y_1$.
We have proved that $[M]$ is in $C"$.

We leave it to the reader to show that these rays are the ``minimal" generators 
of the cone.

Next, we shall prove (3).
Assume that
$\sum_{i = 1}^mq_i[I_i]$ ($ \in G_0(R^{\scriptscriptstyle \# \#})_{\QQ}$)
 is numerically equivalent to $0$.
We may assume that $q_1 \le q_2 \le \cdots \le q_m$.

If $q_1 = q_m$, then $\sum_{i = 1}^mq_i[I_i] = mq_1[R^{\scriptscriptstyle \# \#}]$.
Since $[R^{\scriptscriptstyle \# \#}] \neq 0$ in 
$\overline{G_0(R^{\scriptscriptstyle \# \#})}_{\QQ}$,
we have $q_1 = 0$.

Assume that $q_1 < q_m$.
Then,
\[
\sum_{i = 1}^mq_i[I_i] = q_1m[R^{\scriptscriptstyle \# \#}]
+ \sum_{i = 2}^m(q_i-q_1)[I_i] .
\]
It is easy to see that
\[
\theta^{R^{\scriptscriptstyle \# \#}}(I_i, I_j) = 
\left\{
\begin{array}{lll}
\ell_R(R/(f_i^{a_i}, f_1^{a_1}\cdots f_{i-1}^{a_{i-1}}f_{i+1}^{a_{i+1}}\cdots f_n^{a_n})) & &  (i=j) \\
- \ell_R(R/(f_i^{a_i}, f_j^{a_j})) & &  (i\neq j) ,
\end{array}
\right.
\]
where $\theta^{R^{\scriptscriptstyle \# \#}}$ is the Hochster's
theta pairing.
Then, we have
\[
\theta^{R^{\scriptscriptstyle \# \#}}(I_1, \sum_{i = 1}^mq_i[I_i])
= \sum_{i = 2}^m(q_i-q_1)\theta^{R^{\scriptscriptstyle \# \#}}(I_1,I_i)
= - \sum_{i = 2}^n(q_i-q_1)\ell_R(R/(f_1^{a_1}, f_i^{a_i})) < 0 
\]
because $\theta^{R^{\scriptscriptstyle \# \#}}(I_1, R^{\scriptscriptstyle \# \#})=0$.
By Corollary~6.2 (1) in \cite{DK1}, 
$\sum_{i = 1}^mq_i[I_i]$ is not numerically equivalent to $0$.
\qed

We remark that, if $f \in (x,y)^4S$, then
$R^{\scriptscriptstyle \# \#}$ is not of finite representation type.

If $R$ is a Cohen-Macaulay local ring, the rank of the Grothendieck group
of modules of finite length and finite projective dimension modulo numerical equivalence
coincides with the rank of $\overline{G_0(R)}$ by Proposition~2 in \cite{RS} and 
Theorem~3,1, Remark~3.5 in \cite{K23}.
By (3) of Theorem~\ref{Th6.1}, we know that the rank of the Grothendieck group
of modules of finite length and finite projective dimension modulo numerical equivalence
is equal to $m$ for $R = k[[x_1, x_2, \xi, \eta]]/(\xi\eta + f)$, where 
$0 \neq f = f_1^{a_1}\cdots f_m^{a_m} \in k[[x_1,x_2]]$.

\end{document}